\documentclass[11pt,a4paper]{article}
\usepackage{cmap}
\usepackage[utf8]{inputenc}			
\usepackage[T1]{fontenc}
\usepackage{a4wide}
\usepackage{amssymb,amsmath,amsthm,bm}
\usepackage{mathtools}
\usepackage[english]{babel}
\usepackage{tikz}
\usetikzlibrary{arrows}
\usepackage{graphicx}
\usepackage{enumitem}
\usepackage{makecell}
\usepackage{subcaption}

\usepackage{hyperref}
\usepackage[ruled,vlined]{algorithm2e}

\newtheorem{theorem}{Theorem}

\newtheorem{lemma}[theorem]{Lemma}
\newtheorem{proposition}[theorem]{Proposition}

\newtheorem{remark}[theorem]{Remark}
\newtheorem{conjecture}[theorem]{Conjecture}

\newcommand\pa{{Property $(\ast)$}}

\title{Crumby colorings --- red-blue vertex partition of subcubic graphs regarding a conjecture of Thomassen}

\author{János Barát \\
\small Alfr\'ed R\'enyi Institute of Mathematics\\
\small University of Pannonia, Department of Mathematics\\
\small 8200 Veszprém, Egyetem utca 10., Hungary\\
\small and\\
Zoltán L. Blázsik \\
\small Alfr\'ed R\'enyi Institute of Mathematics\\
\small MTA--ELTE Geometric and Algebraic Combinatorics Research Group \\
\small \url{blazsik@renyi.hu} 
\\
\small and \\
Gábor Damásdi \\
\small MTA-ELTE Lend\"ulet Combinatorial Geometry Research Group, \\ \small  ELTE E\"otv\"os Lor\'and University, Budapest, Hungary
}

\begin{document}
\thispagestyle{empty}
\maketitle


\begin{abstract}
Thomassen formulated the following conjecture: Every $3$-connected cubic graph has a red-blue vertex coloring such that the blue subgraph has maximum degree at most $1$
(that is, it consists of a matching and some isolated vertices) and the red
subgraph has minimum degree at least $1$ and contains no $3$-edge path.
Since all monochromatic components are small in this coloring and there is a certain irregularity, we call such a coloring \emph{crumby}.
Recently, Bellitto, Klimo\v sová, Merker, Witkowski and Yuditsky \cite{counter} constructed an infinite family refuting the above conjecture.
Their prototype counterexample is $2$-connected, planar, but contains a $K_4$-minor and also a $5$-cycle.
This leaves the above conjecture open for some important graph classes: outerplanar graphs, $K_4$-minor-free graphs, bipartite graphs.
In this regard, we prove that $2$-connected outerplanar graphs, subdivisions of $K_4$ and $1$-subdivisions of cubic graphs admit crumby colorings.
A subdivision of $G$ is {\it genuine} if every edge is subdivided at least once.
We show that every genuine subdivision of any subcubic multigraph admits a crumby coloring.
We slightly generalise some of these results and formulate a few conjectures.
\end{abstract}

\section{Introduction}
The graphs in this paper are finite and without loops and multiple edges except Theorem~\ref{t:gen}.
Our notations and terminology mostly follow the Graph Theory book by Bondy and Murty~\cite{b&m}.
In particular, we call a graph {\it subcubic} if it has maximum degree at most 3 and $P_k$ denotes a path on $k$ vertices.

Thomassen \cite{ct17} gave an intricate inductive proof of the following result, a classical case of Wegner's conjecture \cite{weg}: the square of every planar cubic graph is $7$-colorable.
In the same paper, Thomassen formulated an attractive conjecture, which would imply the aforementioned result.

\begin{conjecture} \label{ccc}Every $3$-connected cubic graph has a red-blue vertex coloring such that
the blue subgraph has maximum degree at most $1$
(that is, it consists of a matching and some isolated vertices) and the red
subgraph has minimum degree at least $1$ and contains no $3$-edge path.
\end{conjecture}

Since every monochromatic subgraph is small, but the conditions in the two colors are asymmetric, making this coloring somewhat irregular, we call this a {\em crumby coloring}.
{}From a classical graph decomposition point of view, we are seeking graph classes such that every member admits a crumby coloring.
As it was observed shortly after the appearance of the conjecture, the $3$-prism does not have a crumby coloring. For some time, it looked like this is the only counterexample.
Supporting this, Barát \cite{Barat} showed that every subcubic tree has a crumby coloring.
However, Bellitto et al. \cite{counter} found a construction that produces an infinite family of $3$-connected cubic counterexamples and also $2$-connected planar graphs without crumby colorings.
This fact gives evidence that the intricate induction by Thomassen is somewhat unavoidable.
On the other hand, it leaves open the possibility that crumby colorings might exist for some important graph classes. For instance, outerplanar graphs or bipartite graphs.
Indeed, in Section \ref{s:outer} we show that any $2$-connected subcubic outerplanar graph admits a crumby coloring even if the color of an arbitrary vertex is prescribed.

The fact that we can prescribe the color of a vertex is useful in the  following sense. We believe that crumby colorings exist for every subcubic outerplanar graph.
However, there are various difficulties to extend the results on $2$-connected graphs to all outerplanar graphs.
In a general outerplanar graph there might be trees attached to $2$-connected blocks or between them.
Since Conjecture~\ref{ccc} holds for trees, it gives some hope to combine these two results as building bricks, where having the extra freedom of prescribing the color of a vertex comes into the picture.

The following theorem is a straightforward strengthening of a result of Barát \cite{Barat}. It is routine to check that the original proof literally holds for this version.

\begin{theorem}[Barát \cite{Barat}] \label{t:treesplus}
Every subcubic tree admits a crumby coloring such that the color of a leaf is prescribed.
\end{theorem}

We strengthen this result further in Section \ref{s:outer}.
This allows us to significantly decrease the number of problematic attached trees.

As a weakening of Conjecture~\ref{ccc}, we conjecture that every $K_4$-minor-free graph admits a crumby coloring. This class is interesting for several reasons. Since outerplanar graphs are $K_4$- and $K_{2,3}$-minor-free, this would be a natural extension step from outerplanar graphs. It also concurs with the fact that all known counterexamples to Conjecture~\ref{ccc} contain $K_4$-minors. In contrast, we show a crumby coloring of any subdivision of $K_4$ in Section~\ref{sec:k4}.


However, we first prove that members of a special class of bipartite graphs admit crumby colorings in Section~\ref{sec:bip}.
They are the $1$-subdivisions of cubic graphs, that arise from cubic graphs by adding an extra vertex on each edge.
In this way, we form bipartite graphs, where the vertices in one class have degree 2 and in the other class degree 3. A crucial idea in the proof is to use the maximum matching of the original graph. To this end, we employ the famous Edmonds-Gallai decomposition theorem.

Motivated by the previous results, we introduced the notion of \emph{genuine subdivision} of a graph $G$, that is a graph $H$, which we get from $G$ by subdividing every edge of $G$ by at least one vertex. As a generalization, in the latter part of Section~\ref{sec:bip}, we prove that every genuine subdivision of any subcubic multigraph admits a crumby coloring. We state this result for multigraphs because at one point in the proof, we eliminate the degree 2 vertices and this step can create parallel edges.

\section{Bipartite graphs and subdivisions\label{sec:bip}}

Despite the infinite family of counterexamples in \cite{counter}, we still believe that Conjecture~\ref{ccc} holds for most subcubic graphs.
We pose the following conjecture.

\begin{conjecture} \label{bip}
Every subcubic bipartite graph admits a crumby coloring.
\end{conjecture}

We can prove this for a special class of bipartite graphs, where the degrees are all 2 in one class and 3 in the other class.
In the proof, we apply the Edmonds-Gallai decomposition theorem \cite{edm,gal} that gives us information about the structure of the maximum matchings of a graph $G$.
We recall that $P_k$ denotes a path with $k$ vertices and $N(X)$ denotes the set of neighbors of a vertex set $X$.
A graph $G$ is {\it hypomatchable} or {\it factor-critical} if for every vertex $x$, the graph $G-x$ has a perfect matching.

\begin{theorem}[Edmonds-Gallai decomposition]\label{e-g}
Let $G$ be a graph and let $A\subseteq V(G)$ be the collection of all vertices $v$ such that there exists a maximum size matching which does not cover $v$.
Set $B=N(A)$ and $C=V(G)\setminus (A \cup B)$. Now
\begin{enumerate}
    \item[$(i)$] Every odd component $O$ of $G-B$ is hypomatchable and $V(O)\subseteq A$.
    \item[$(ii)$] Every even component $Q$ of $G-B$ has a perfect matching and $V(Q)\subseteq C$.
    \item[$(iii)$] For every $X\subseteq B$, the set $N(X)$ contains vertices in more than $|X|$ odd components of $G-B$.
\end{enumerate}
\end{theorem}


In what follows, we study subdivisions of cubic graphs.
If we add precisely one new vertex on each edge, then the resulting graph is a {\it $1$-subdivision}.
We support Conjecture~\ref{bip} by showing the following

\begin{theorem} \label{t:1-sub}
Let $S(G)$ be the $1$-subdivision of a cubic graph $G$.
The bipartite graph $S(G)$ admits a crumby coloring.
\end{theorem}

\begin{proof}
The idea of the proof is to color the original vertices (in $G$) red and color the subdivision vertices blue.
If $G$ admits a perfect matching $M$, then we recolor the subdivision vertices on $M$ to red.
This results in a crumby coloring consisting of red $P_3$-s and blue singletons.
We refer to this idea later as the {\em standard process}.
For instance, every 2-edge-connected graph $G$ admits a perfect matching by Petersen's Theorem.
If the graph $S(G)$ is the 1-subdivision of such $G$, then the standard process gives a crumby coloring of $S(G)$.

In what follows, we modify this simple idea to the general case, where $G$ is any cubic graph.
If $G$ does not possess a perfect matching, we can still consider a maximum size matching in $G$ and use the Edmonds-Gallai decomposition.

Let $G$ be a cubic graph, and let $B$ be given by the Edmonds-Gallai decomposition.
Any isolated vertex in $B$ must be connected to at least two odd components of $G-B$.
The third edge might go to a third odd component, an even component or to one of the first two odd components.

Initially, let every vertex of $G$ be red and
every subdivision vertex blue.
We recolor a few vertices as follows.
In every even component, there exists a perfect matching and we recolor the subdivision vertices on the matching edges to red.

Consider the vertex sets $A$ and $B$ of the Edmonds-Gallai decomposition.
Contract the components of $A$ to vertices to get $A^*$.
The bipartite graph $(A^*,B)$ satisfies the Hall-condition by property $(iii)$. Therefore, we find a matching $M$ covering $B$.
We recolor the subdivision vertices of the matching edges in $M$ to red.
We continue with the odd components corresponding to the vertices of $A^*$ saturated by $M$.
In these components, we use property $(i)$ and find an almost perfect matching (if it is needed because the size of this component is greater than 1).
The subdivision vertices on these matching edges are colored red as well.
So far we only created red $P_3$-s separated by blue singletons.
What is left to consider is the union of odd components corresponding to unsaturated vertices of $A^*$.

Let $H$ be an odd component, which is a single vertex $x$. The $G$-neighbors of $x$ are in the set $B$. Suppose $y$ is a neighbor of $x$.
There are two different types of $y$-vertices depending on the location of its 3 neighbors.
Vertex $y$ is {\em ordinary} if it has no $G$-neighbor in $B$.
Vertex $y$ is {\em problematic} if it has precisely one $G$-neighbor in $B$.
Notice that a vertex $y$ in $B$ cannot have at least two $G$-neighbours in $B$ by property $(iii)$.

Assume $y$ is ordinary, and $x$ is a singleton odd component.
By the above coloring, vertex $y$ belongs to a red $P_3$ since $B$ was saturated, and $y$ has two blue subdivision vertices as neighbors in $S(G)$.
We recolor the subdivision vertex $v_{xy}$ on the edge $xy$ red and $y$ blue.
If the third $G$-neighbor $w$ of $y$ belongs to either an even component or a saturated odd component, then it already has a red neighbor and causes no trouble.
Notice that $w$ might be a singleton odd component or it may belong to an unsaturated odd component as well.
If $w$ is another singleton odd component, then we recolor $v_{wy}$ red and $y$ remains a singleton blue component, which in turn further decreases the number of isolated red vertices.
However, if $w$ belongs to a larger unsaturated odd component, then again we recolor $v_{wy}$ red and finish the coloring of this odd component as it was explained before for the larger saturated odd components using property $(i)$.

We perform this coloring step for all of those unsaturated singleton odd components, which have an ordinary neighbor from $B$. Observe that at this point among the colored vertices of $G$ the blue ones must be some ordinary vertices of $B$ and all three $G$-neighbors of these vertices are red and has
a red neighbor.

Continue the recoloring process by considering one-by-one the unsaturated odd singleton components, which only have problematic $G$-neighbors from $B$. Assume $x$ is such a singleton odd component, and $y$ is a problematic $G$-neighbor of $x$.
We recolor $v_{xy}$ red and $y$ blue.
Since $y$ has a unique $G$-neighbor $y'$ in $B$, we consider the third $G$-neighbor $x'$ of $y'$ (besides $y$ and the one determined by $M$).
If $x'$ already has a red neighbor, then we are done and we can continue the recoloring process. Otherwise $x'$ must belong to an odd component which still has isolated red vertices.

If $x'$ is an isolated red vertex and belongs to a large odd component $H'$, then denote one of its $G$-neighbors inside $H'$ by $z'$. Use property $(i)$ in $H'-z'$ and fix a perfect matching there and color the subdivision vertices red on these matching edges. Now recolor $v_{x'z'}$ red and $x'$ blue. This way $z'$ is no longer an isolated red and the blue component of $x'$ consists of $x'$ and $v_{x'y'}$.

If $x'$ is an isolated red vertex and belongs to a singleton odd component which only has problematic $G$-neighbors from $B$ then we continue the recoloring process with $x'$. Since $x'$ has degree 3 in $G$, we can select a $G$-neighbor $z$ different from $y'$ and do as above.
In this way, the color of $y'$ is unchanged and therefore $y$ remains in a blue $P_2$.
Altogether we created a crumby coloring locally around $x$ and $y$.
Now we continue with $x'$ and $z$ playing the role of $x$ and $y$ in the previous argument.
This process terminates and creates no loops, since every $B$-vertex is incident to 3 edges, one of which belongs to $M$. Therefore this process have to end by finding a vertex from one of the odd components which already had a red neighbor. Let us emphasize that this process cannot go back to any of the unique blue vertices of some large odd component because we cannot revisit the already visited problematic vertices of $B$.

At this point we either have a crumby coloring or there are some unsaturated large odd components which haven't been visited during the recoloring process.

Let $H$ be such a large odd component and $x \in H$ be an arbitrary vertex and consider a perfect matching in $H-x$ by property $(i)$. We recolor the subdivision vertices on these matching edges to red. Let $y$ be a $G$-neighbor of $x$ in $H$ and recolor $v_{xy}$ be red and $y$ blue. Since there was a matching edge $zy$ and both $z$ and $v_{yz}$ are red, moreover on the third edge $wy$ in $G$ incident to $y$, the subdivision vertex is blue but $w$ must be red. Indeed, since $w$ cannot be blue if $w\in H$ and if $w\in B$ then it cannot be blue because in that case we must have already considered this edge $wy$ and thus $H$ cannot be an non-visited large odd component. Hence $y$ and $v_{wy}$ form a blue $P_2$ together.

We can recolor all the remaining non-visited large odd components by the same argument. After all these steps a crumby coloring of $S(G)$ arises. \end{proof}

Next, we complement the previous result.
Here we allow all longer subdivisions.

\begin{lemma}
Let $G$ be a cubic graph.
Let $H$ be an arbitrary subdivision of $G$ such that every edge is subdivided at least twice.
The graph $H$ admits a crumby coloring.
\end{lemma}

\begin{proof}
Let us color the original vertices of $G$ blue.
We find that almost any subdivided edge admits a crumby coloring such that the end-vertices are singleton blues.
The only exception is the case with 4 subdivision vertices.
In particular, we use the following colorings for the internal vertices ($r$,$b$ stands for red and blue, respectively):
$rr$, $rrr$, $rrrb$, $rrbrr$, $rrrbrr$, $rrbbrrr$, $rrbrrbrr$ etc. 

Let us use these particular colorings on $H$.
We might create some blue stars with 2 or 3 leaves.
Apart from that, this coloring satisfies the crumby conditions.
Now we recolor the problematic blue centers of these stars red.
If the vertex $c$ is such a center, and there was a blue 3-star at $c$, then we recolor the neighbor $n_1$ of $c$ red and recolor the neighbor $n_2$ of $n_1$ blue.
If vertex $c$ was the center of a blue 2-star, then we have to consider two cases according to the red neighbor $v$ of $c$.
If $v$ was the end-vertex of a red $P_3$, then we do the same recoloring as in the previous case, but also recolor $v$ to blue.
If $v$ was the end-vertex of a red $P_2$, then the recoloring of $c$ creates a red $P_3$ and we are done.

The process terminates with a crumby coloring of $H$.
\end{proof}

Motivated by the results of this section, we prove the existence of crumby colorings in a generalized setting, namely for genuine subdivisions of subcubic multigraphs. In order to prove a generalization of Theorem~\ref{t:1-sub}, we use a different approach, and to make the proof more transparent we state our result for multigraphs.

\begin{theorem} \label{t:gen}
Every genuine subdivision $S(G)$ of any subcubic multigraph $G$ admits a crumby coloring.
\end{theorem}

\begin{proof}
We may assume the considered graphs are connected, otherwise we can repeat the same argument on each connected component. The graphs are also loopless, since loop edges do not make any difference regarding the conditions of crumby colorings.

Suppose to the contrary there exists a connected subcubic loopless multigraph and a corresponding genuine subdivison, which does not admit a crumby coloring. Let $G$ denote the smallest such graph, i.e. $G$ has the least number of vertices and among those the least number of edges. Fix such problematic pair $(G,S(G))$.

As the first step, we summarize crumby colorings of paths $P_k$ on $k$ vertices for different purposes, which we use later. In Table \ref{tab:utszin}, we highlighted by capital letters those cases, in which the corresponding aim in the header is not attainable. Notice for $k\ge 8$ all of these goals are achievable because from $k\ge11$ one can get a crumby coloring with a certain goal of $P_k$ by extending the crumby coloring of $P_{k-3}$ between two subdivision vertices of different colors with $rrb$. Also if the two end-vertices are required to be singletons regardless of their colors, then we can achieve this already for $k\ge~7$.

\begin{table}[h!]
    \centering
\[
\begin{array}{|c|c|c|c|c|c|c|}
    \hline
    k & \makecell{\mathrm{one ~endpoint} \\ \mathrm{is~a~singleton} \\ \mathrm{blue,~and} \\ \mathrm{the~other~is}\\ \mathrm{in~a~red~} K_2} & \makecell{\mathrm{both} \\ \mathrm{endpoints~are} \\ \mathrm{in~a~red~} K_2} & \makecell{\mathrm{one ~endpoint} \\ \mathrm{is~a~singleton} \\
    \mathrm{red,~and} \\ \mathrm{the~other~is}\\ \mathrm{in~a~red~} K_2} & \makecell{\mathrm{both} \\
    \mathrm{endpoints} \\ \mathrm{are~red} \\ \mathrm{singletons}} & \makecell{\mathrm{both} \\
    \mathrm{endpoints} \\ \mathrm{are~blue} \\ \mathrm{singletons}} & \makecell{\mathrm{one ~endpoint} \\
    \mathrm{is~a~singleton} \\
    \mathrm{red,~and~the} \\ \mathrm{other~is~a}\\ \mathrm{singleton~blue}} \\
    \hline
    3 & rrb & RBR & RBR & rbr & RRB & RRB \\
    \hline
    4 & RRBB & RRBR & rrbr & rbbr & brrb & RRBB \\
    \hline
    5 & RRBBR & rrbrr & rrbbr & RRBBR & brrrb & rbrrb \\
    \hline
    6 & rrbrrb & rrbbrr & RRBBRR & rbrrbr & BBRRRB & rbbrrb \\
    \hline
    7 & rrbbrrb & RRBRRBR & rrbrrbr & rbrrrbr & brrbrrb & rbbrrrb \\
    \hline
    8 & rrbbrrrb & rrbrrbrr & rrbbrrbr & rbrrrbbr & brrbbrrb & rbrrbrrb \\
    \hline
    9 & rrbrrbrrb & rrbbrrbrr & rrbbrrbbr & rbbrrrbbr & brrrbbrrb & rbbrrbrrb \\
    \hline
    10 & rrbbrrbrrb & rrbbrrrbrr & rrbbrrrbbr & rbrrbrrbbr & brrrbbrrrb & rbbrrbbrrb \\
    \hline
\end{array}
\]
\caption{Crumby colorings of $P_k$ for different purposes.}
    \label{tab:utszin}
\end{table}

Before the coloring step of the proof, which leads to a contradiction, we need to observe some structural properties of $G$ and $S(G)$.

We may assume in the genuine subdivison $S(G)$ every edge is subdivided by at most 4 vertices.
Otherwise, we can delete an edge $uv$ with at least 5 subdivision vertices, and consider a crumby coloring of the remaining subdivided graph $(G',S(G'))$.
It can be completed to a crumby coloring of $S(G)$ by using one of the last three columns of Table \ref{tab:utszin} depending on the given colors of $u$ and $v$.

We may assume the minimum degree of $G$ is at least 2. Otherwise, after the deletion of a vertex $v$ of degree 1, the remaining subdivided graph $(G',S(G'))$ admits a crumby coloring by assumption. This coloring can be completed to a crumby coloring of $S(G)$ by elementary considerations.

Moreover, we claim $G$ is a 3-regular multigraph.
Suppose to the contrary there is a vertex $v$ of degree 2. There are either one or two neighbors of $v$, and we deal with these two cases separately.

{\it Case a1):} Suppose $v$ has only one neighbor $u$ and two parallel edges between them.
If $\deg(u)=2$, then this is the whole graph $G$.
In that case, $S(G)$ is a cycle of length at least 4, thus it has a crumby coloring as Table \ref{tab:cycle} shows.
It is also true that a cycle has a crumby coloring even if it has an arbitrary vertex with a prescribed color since we can rotate the crumby colorings of Table \ref{tab:cycle} accordingly.

\begin{table}[!h]
    \begin{center}
\[
    \begin{array}{|c|c|c|c|c|c|c|}
    \hline
    k & 3 & 4 & 5 & 6 & 7 & 8 \\
    \hline
    \mathrm{crumby~coloring~of~}C_k & rrb & rrrb & rrrbb & rrbrrb & rrbrrrb & rrrbrrrb \\
    \hline
    \end{array}
\]
    \end{center}
    \caption{For larger $k$, start with $rrb$ and continue with the crumby coloring of $C_{k-3}$.}
    \label{tab:cycle}
\end{table}

Suppose $\deg(u)=3$, and the other neighbor of $u$ is $w$. Consider the induced subgraph $H$ of $S(G)$ on $V\setminus\{u,v\}$. By the indirect assumption, $H$ admits a crumby coloring.
Therefore, we assume the color of $w$ is fixed.
We use the opposite color on the closest subdivision vertex along $uw$.
If $w$ is blue, then color the vertices along $uw$ by the sixth column of Table \ref{tab:utszin}.
However, if $w$ is red, then use the first column except for $k=4$, in which case we use $rrbr$.
Let us mention if $uw$ has only one subdivision vertex, then the color of the subdivision vertex is determined, it differs from the color of $w$, and the color of $u$ is red. 

At this point, $u$ is colored and it is either a singleton blue or a singleton red, or belongs to a red $K_2$, hence we can finish the crumby coloring of $S(G)$ by rotating appropriately the coloring of the cycle formed by the two parallel edges between $u$ and $v$. Note that even if $u$ belongs to a red $K_2$ along $uw$, and the component of $u$ on the cycle is a red $P_3$, we can rotate the cycle so that $u$ becomes the middle vertex of the red $P_3$.
Therefore, we can avoid creating a red $P_4$. We get a crumby coloring of $S(G)$, which is a contradiction.

{\it Case a2):} Suppose $v$ has two distinct neighbors $u$ and $w$. Let $G'$ denote the graph, which we get from $G$ by deleting $v$, and adding an edge $uw$. Note that this may create parallel edges (this is the reason why we consider multigraphs instead of simple graphs). Consequently, there is a genuine subdivision $S(G')$ of $G'$, which is isomorphic to $S(G)$, but that contradicts the assumption that $G$ was a smallest counterexample.

Hence $G$ must be a $3$-regular multigraph. However, we claim that actually $G$ is a $3$-regular (simple) graph.
Suppose to the contrary there is a vertex $v$ incident to some parallel edges.

{\it Case b1):} Suppose $v$ has only one neighbor $u$, and there are three parallel edges between them, so it is the whole graph $G$. Let the 3-tuple $(x,y,z)$ denote the number of subdivision vertices on the three edges, and assume $x\le y\le z$. If $2\le x$, then color $u$ and $v$ blue and use the fifth column of Table \ref{tab:utszin} on each edge.
In this way, we get a crumby coloring except for $x=y=z=4$.
However, that specific graph admits a crumby coloring as illustrated in Figure \ref{fig:3par}.

If $x=1$, then color $u,v$ red and the single subdivision vertex between them blue.
We use the second column of Table \ref{tab:utszin} for the other two edges. Thus we get a crumby coloring, unless $y=1$ and $z\in\{1,2,5\}$. In all other cases both $u$ and $v$ get a red neighbor.
The remaining three specific graphs also have crumby colorings, see Figure \ref{fig:3par}.

\begin{figure}[!h]
    \centering
    \includegraphics[width=0.9\textwidth]{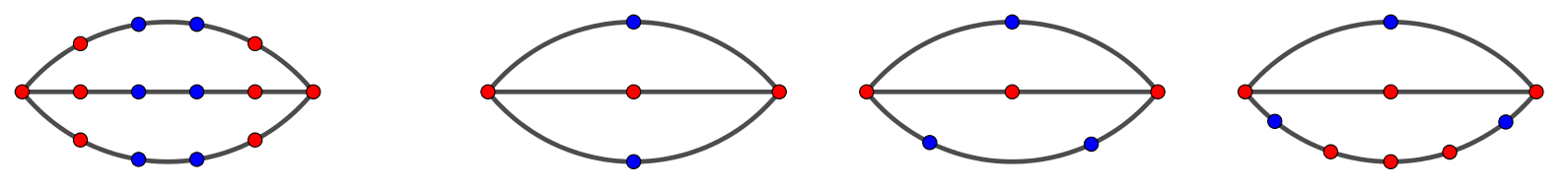}
    \caption{Crumby colorings for the following cases: $(4,4,4)$ and $(1,1,z)$ for $z\in\{1,2,5\}$}
    \label{fig:3par}
\end{figure}

{\it Case b2):} Suppose $v$ has exactly two neighbors $u$ and $w$, and assume $u$ is adjacent to $v$ by two parallel edges.
The third edge incident to $u$ goes either to $w$ or another vertex $q$ (see Figure \ref{fig:2par}). In the former case, let $G'$ denote the graph, which we get from $G$ by deleting $u,v,w$. In the latter case, let $G''$ denote the graph, which we get from $G$ by deleting $u$ and $v$.
Fix the crumby colorings of $S(G')$ and $S(G'')$ existing by the induction hypothesis. The contradiction arises by extending these crumby colorings to crumby colorings of $S(G)$ as follows.

\begin{figure}[!h]
    \centering
    \includegraphics[width=0.6\textwidth]{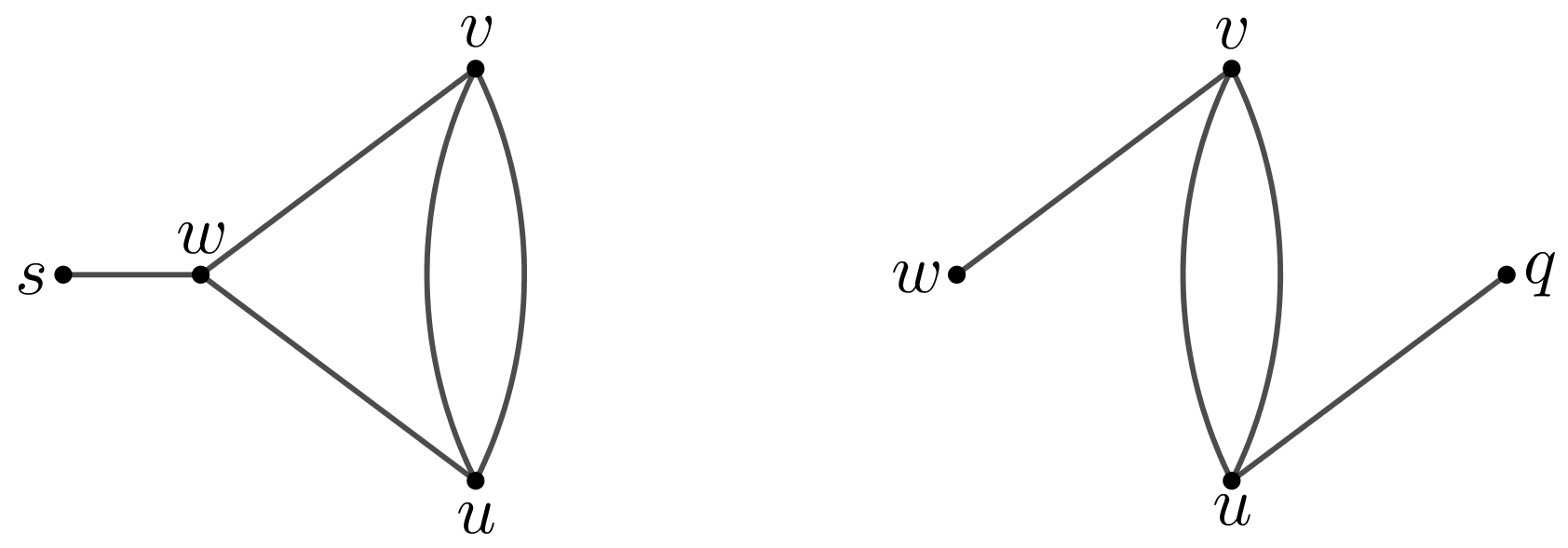}
    \caption{The two possible situations in {\it Case b2)}}
    \label{fig:2par}
\end{figure}

We start the extension by changing the color of the first subdivision vertex next to $s$ and in the other case next to $w$ and $q$.
We continue repeating $rrb$ periodically along the edges $sw$, $wv$ and $wu$ and in the other setup along $wv$ and $qu$. Thus each of $u$ and $v$ becomes either a red vertex in a $P_2$ component or a singleton in any color.
The subdivided parallel edges between $u$ and $v$ can be considered as a cycle of length between 4 and 10 with prescribed colors on two non-adjacent vertices.
Hence in Table \ref{tab:pre2cycle}, we listed the crumby colorings for this cycle with respect to the prescribed colors of $u$ and $v$. Note that extending the fixed colorings of $S(G')$ and $S(G'')$ by the colorings in Table \ref{tab:pre2cycle} becomes crumby only if $u$ and $v$ are singletons in some color. Even in that subcase, there are some exceptions that we leave for later. These instances are striked out in the table. 

\begin{table}[!h]
    \begin{center}
\[
    \begin{array}{|c|c|c|c|}
    \hline
    \makecell{\mathrm{number~of~subdivision} \\ \mathrm{vertices~on~the}\\ \mathrm{parallel~edges}} & u,v \mathrm{~are~both~blue} & \makecell{u \mathrm{~is~blue,~} v \mathrm{~is~red} \\ \mathrm{or~vice~versa}} & u,v \mathrm{~are~both~red} \\
    \hline
    (1,1) & - & BbRr & RrRb~^{\ast} \\
    \hline
    (1,2) & - & BrRrb & RrRbb~^{\ast} \\
    \hline
    (1,3) & - & BrRbrr & RbRrbr \\
    \hline
    (1,4) & - & BrRbbrr & RbRrbbr \\
    \hline
    (2,2) & BrrBrr & -~^{\ast} & RrbRrb \\
    \hline
    (2,3) & BrrBrrr & BbrRbrr & RbbRrbr \\
    \hline
    (2,4) & BrrBbrrr & BbrRbrrr & RbbRrbbr \\
    \hline
    (3,3) & BrrrBrrr & BrrbRrrb~^{\ast} & RrbbRrbb \\
    \hline
    (3,4) & BrrrBbrrr & BrrbRrbrr & RrbrRbrrb \\
    \hline
    (4,4) & BbrrrBbrrr & BrrbbRrbrr & RrbbrRbrrb \\
    \hline
    \end{array}
\]
    \end{center}
    \caption{Crumby colorings of cycles of small length with prescribed colors on two non-adjacent vertices (capital letters corresponds to the color of $u$ and $v$).}
    \label{tab:pre2cycle}
\end{table}

However, if any of $u$ and $v$ is in a red $P_2$ component, then the extended colorings remain crumby unless the corresponding vertex is the end of a red $P_3$ in those colorings, which are marked by $~^{\ast}$ in Table \ref{tab:pre2cycle}. Observe if $u$ or $v$ is prescribed blue (the striked out instances), then we are allowed to change its color to red unless $w$ is a common neighbor of $u$ and $v$ in $G$, and both $wu$ and $wv$ has only 1 subdivision vertex and $w$ is red. In this special case, we change the color of the only subdivision vertex on $wv$ to blue and thus $v$ gets red and we can continue with the second column of Table~\ref{tab:pre2cycle} (without changing the prescribed blue color of $u$). On the flip side, if $u$ or $v$ is prescribed to be a singleton red (in certain instances marked by $^{\ast}$), then we are free to change its color to blue. Using these observations, the remaining cases can be finished in an alternate way. The remaining cases are shown in Figure \ref{fig:col1},\ref{fig:col2} and \ref{fig:col3} separately for the three columns of Table \ref{tab:pre2cycle}.

\begin{figure}[!h]
    \centering
    \includegraphics[width=0.8\textwidth]{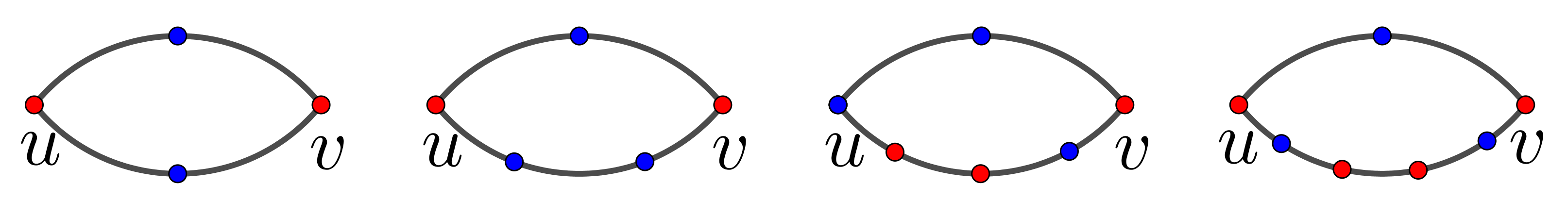}
    \caption{In these graphs both $u$ and $v$ are prescribed to be blue, however we must change color for at least one of them to red creating a red $P_3$ component.}
    \label{fig:col1}
\end{figure}

\begin{figure}[!h]
    \centering
    \includegraphics[width=0.7\textwidth]{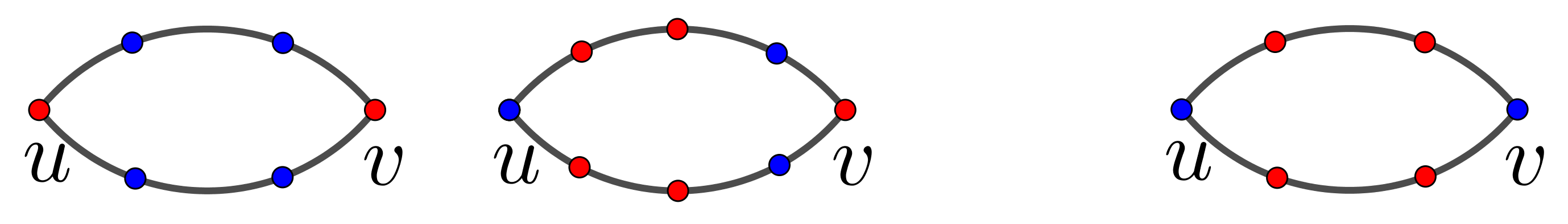}
    \caption{In the two graphs on the left $u$ is prescribed to be blue, and $v$ in a red $P_2$ component, while in the graph on the right $u$ is prescribed to be blue, and $v$ a singleton red.}
    \label{fig:col2}
\end{figure}

\begin{figure}[!h]
    \centering
    \includegraphics[width=0.8\textwidth]{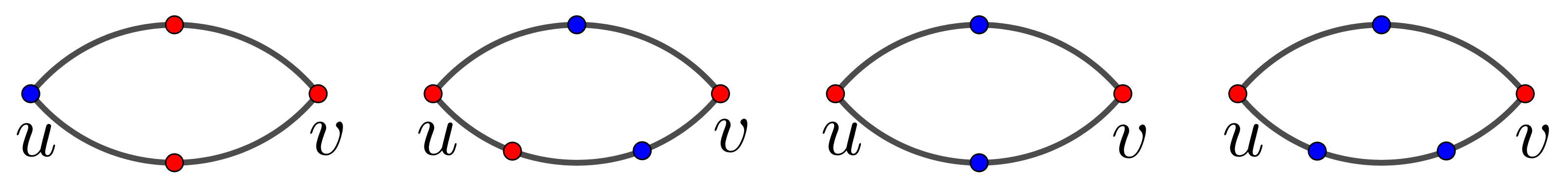}
    \caption{In the two graphs on the left $u$ is prescribed to be a singleton red, and $v$ in a red $P_2$ component, while in the two graphs on the right both of them are prescribed to be in a red $P_2$ component.}
    \label{fig:col3}
\end{figure}

Thus (to avoid the previous contradictions) $G$ must be a $3$-regular (simple) graph. Here comes the coloring part of the proof.
Let us call the vertices of $G$ {\it original} in the subdivided graph $S(G)$.
We start the coloring process by using the second column of Table \ref{tab:utszin} for those edges, which have either 3 or 4 subdivision vertices. At this point, all  colored original vertices are red and already have a red neighbor.
Also, every uncolored original vertex is incident to 3 edges, each of which contains at most 2 subdivision vertices.

Consider the uncolored original vertices one-by-one, which are incident to at least two edges with exactly 1 subdivision vertex.
Such an original vertex $v$ gets color blue, and we use the first column of Table \ref{tab:utszin} on the three incident edges.
Hence all its $G$-neighbors are red and have a red neighbor. Note that from now on there are both red and blue original vertices.
It still holds that the red ones have a red neighbor. 
Furthermore, the uncolored original vertices are not adjacent (in $G$) to any of the blue original vertices. Perform this step repeatedly until there are only uncolored original vertices, which are incident to at least two edges with exactly 2 subdivision vertices.

Consider now uncolored original vertices, which are incident to one edge with exactly 1 subdivision vertex. Assume that $v$ is such a vertex, and let $vx$ be the edge with exactly 1 subdivision vertex. If $x$ is uncolored, then both other edges incident to $x$ have exactly 2 subdivision vertices.
This means the $G$-neighbors of $v$ and $x$ are either uncolored or red (which already have a red neighbor). If the $G$-neighbor is red, then we use $rbbr$ on the corresponding edge.
If the $G$-neighbor is uncolored, then use $rrbr$ so that the uncolored $G$-neighbor gets the red neighbor.
We color the only subdivision vertex on $xv$ red.
Now $x$ and $v$ become the endpoints of a red $P_3$ component, but all of their $G$-neighbors are red and have a red neighbor. 


If $x$ has already been colored red, then we color the only subdivision vertex on $xv$ blue.
If at least one of the other two $G$-neighbors of $v$ is also red, then by using $rrbr$ or $rbrr$ appropriately on the two edges, we can provide that $v$ and the possible uncolored $G$-neighbor of $v$ are red and have a red neighbor.
Meanwhile the components of the already red original vertices do not grow.
If both other $G$-neighbors of $v$ are uncolored, then color $v$ blue and use $rrrb$ on these edges.
In this step, we might create red original vertices, which are at the end of some red $P_3$ components.
However, let us emphasize that the $G$-neighbors of the blue original vertices are still  red and have a red neighbor.

Hence the uncolored original vertices must have 3 incident edges with exactly 2 subdivision vertices.
If at least one of the $G$-neighbors of $v$ has been colored (by the previous observation its color must be red), then use $rrbr$ or $rbrr$ appropriately again on these edges to make sure that $v$ and the possible uncolored $G$-neighbors of $v$ become red and have a red neighbor.
Meanwhile the components of the already red original vertices do not grow. On the other hand, if the $G$-neighbors of $v$ are still uncolored, then color $v$ blue and use $rrrb$ on these edges.

The coloring is almost finished, since every original vertex is colored, but there might be some edges which have not been considered, yet.
Throughout the coloring process if an original vertex got blue, then all subdivision vertices of incident edges got colored at the same time, and its $G$-neighbors became red.
Thus if an edge has not been considered yet, then it has at most 2 subdivision vertices and both of its endpoints are red. Hence we can complete the coloring with either $b$ or $bb$ at these subdivision vertices.

The coloring process terminates in a crumby coloring of $(G)$, that is a contradiction.
\end{proof}

\section{Outerplanar graphs} \label{s:outer}

We know that Conjecture~\ref{ccc} holds for trees and fails in general for $2$-connected planar graphs.
A natural minor-closed class between the aforementioned classes is the class of outerplanar graphs.
As the first step, we prove the following.

\begin{theorem}\label{t:outer2}
Let $G$ be a $2$-connected subcubic outerplanar graph and let $v$ be a vertex of $G$. We may prescribe the color of $v$ and find a crumby coloring of $G$.
\end{theorem}


\begin{proof}
We consider $G$ together with its outerplanar embedding. An ear decomposition of a $2$-connected graph $G$ is a series of graphs $G_0, G_1, \dots, G_k=G$ such that $G_0$ is a cycle and $G_i=G_{i-1}\cup E_i$, where each $E_i$ is a path that have its two leaves in common with $G_{i-1}$. We may assume that $G_0$ is a bounded face containing the vertex $v$, and if $d(v)=3$ then let $v$ be an endpoint of $E_1$.
Since $G$ is a 2-connected outerplanar graph, it has an open ear decomposition such that on each ear the attachment vertices (endpoints) are adjacent.
The endpoints of the ears are different by the subcubic property.

\begin{figure}[!h]
    \centering
    \includegraphics[width=0.4\textwidth]{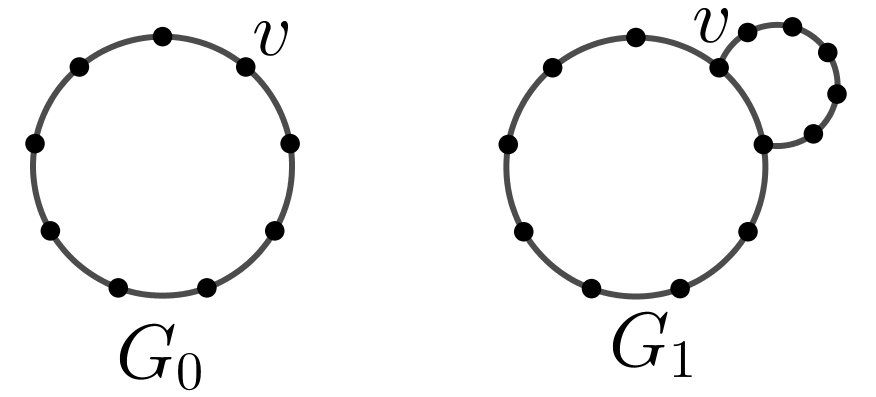}
    \caption{Starting points for the coloring process depending on the degree of $v$.}
    \label{fig:g0g1}
\end{figure}

In general, we start the coloring process with $G_0$.
There is an exceptional situation though.
If $d(v)=3$, then we immediately add the other bounded face containing $v$ as the first ear to form $G_1$ (see Figure \ref{fig:g0g1}).
We first show that the starting subgraph ($G_0$ or $G_1$ depending on the degree of $v$) of $G$ has a crumby coloring.
Secondly, we show that if $G_i$ has a crumby coloring, then $G_{i+1}$ also admits a crumby coloring in which the colors of the vertices of $G_i$ are unchanged except possibly the endpoints of the ear $E_{i+1}$.
This procedure leads to a crumby coloring of $G$.
During the coloring process, we establish and maintain a significant property of our crumby coloring.
Namely, we never color two adjacent vertices of degree 2 (with respect to the current subgraph) blue, unless we know that there is no later ear with this pair of endpoints. Let us call it \pa.

We use the shorthand $r$ for red and $b$ for blue.

\bigskip

{\bf Starting the procedure:}~If $d(v)=2$, then the only bounded face of $G$ containing vertex $v$ is a cycle of length $k\ge 3$. We know that $G_0=C_k$ has a crumby coloring, but we need more.
For $k=5$, we must observe that there exist adjacent vertices $x$ and $y$ in $C_5$, which are not the endpoints of any ear\footnote{Vertex $x$ might be the endpoint of ear $E_i$, but in that case $y$ is not the other endpoint.}.
Therefore, we color $x$ and $y$ blue and the remaining 3 vertices red in order to establish \pa. The required crumby colorings of $C_k$ was previously shown in Table \ref{tab:cycle}.


If $k\ne 5$, then we can rotate the above given crumby colorings such that $v$ gets its prescribed color.
We notice the following for $k=5$.
If the prescribed color of $v$ is red, then we can choose two adjacent vertices $x$ and $y$ of the $5$-face containing $v$ (distinct from $v$), for which there is no (later) ear connecting them.
We color $x$ and $y$ blue and the rest red to establish \pa.
If $v$ is supposed to be blue, then \pa\ holds immediately as we rotate the given coloring of $C_5$ to make $v$ blue.

If $d(v)=3$, then we show a crumby coloring of $G_1$, the subgraph spanned by the two bounded faces containing $v$.
One endpoint of the first ear $E_1$ is $v$.
We denote the other endpoint by $u$.
Suppose that the boundary of $G_0$ is $(u, w_1, w_2, \dots, w_k, v)$, where $u$ and $v$ are adjacent, and the internal points of $E_1$ are $(z_1,z_2,\dots,z_{\ell})$ from $u$ to $v$.
Firstly, we give a crumby coloring for $k=\ell=2$.
There are two cases depending on the prescribed color of $v$, see Figure \ref{kl2}.

\begin{figure}[!ht]
    \centering
    \includegraphics[width=0.6\linewidth]{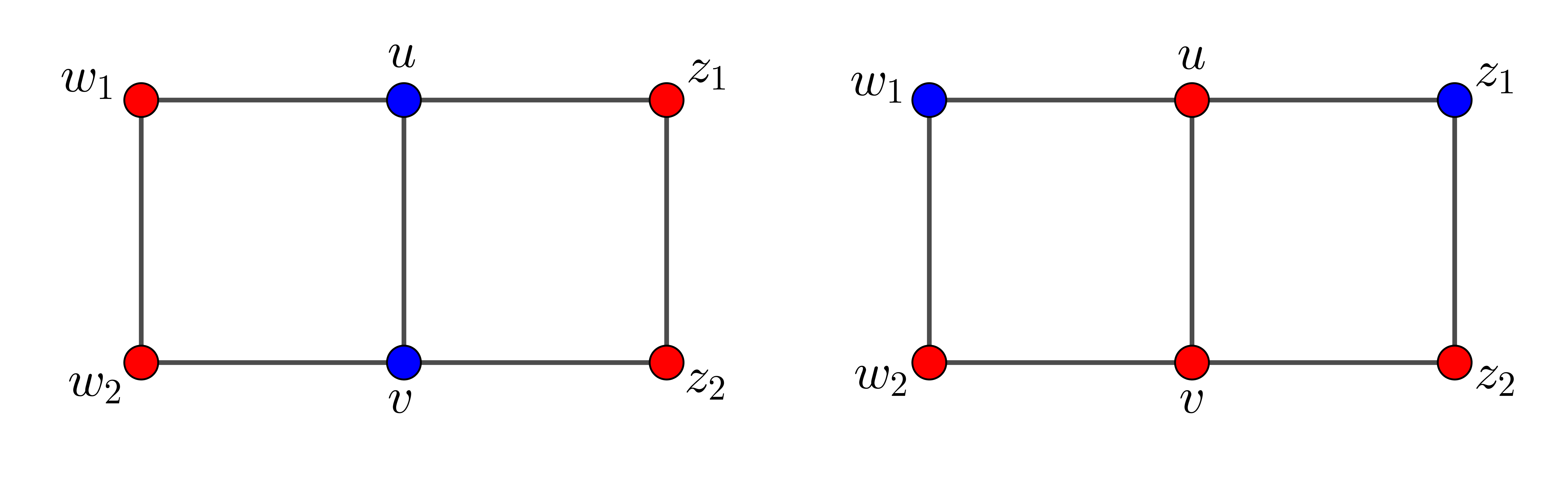}
    \caption{Crumby colorings for $k=\ell=2$.}
    \label{kl2}
\end{figure}

In the remaining cases, we color $u$ and $v$ differently and
assume the prescribed color of $v$ to be red.
If it was blue, then $u$ plays the role of $v$. In the following table, we summarize the initial colorings depending on the values of $k$ and $\ell$.

\begin{table}[!ht]
\begin{center}
\[
\begin{array}{|c|c|c|c|}
    \hline
    ~ & \makecell{\ell=3d+1 ~(d \in \mathbb{N}) \\ \mathrm{the~color~of~} \\ (z_1,z_2,\dots,z_{\ell})} & \makecell{\ell=3d+2 ~(d \in \mathbb{N}) \\ \mathrm{the~color~of~} \\ (z_1,z_2,\dots,z_{\ell})} & \makecell{\ell=3d+3 ~(d \in \mathbb{N}) \\ \mathrm{the~color~of~} \\ (z_1,z_2,\dots,z_{\ell})} \\
    \hline
    \makecell{k=3c+1 ~ (c\in \mathbb{N}) \\ \mathrm{let~the~color~of~} \\ (w_1,w_2,\dots,w_k) \\ \mathrm{be~} (rrb)^{c}r} & (rrb)^{d}r & \makecell{\mathrm{if~}d=0:~ br\\ \mathrm{if~}d\ge1:~ (rrb)^{d-1}rrrbr} & (rrb)^{d}rrb \\
    \hline
    \makecell{k=3c+2 ~ (c\in \mathbb{N}) \\ \mathrm{let~the~color~of~} \\ (w_1,w_2,\dots,w_k) \\ \mathrm{be~} \underbrace{(rrb)^{c-1}rrr}_{\mathrm{if~} c~\ge~1}br} & (rrb)^{d}r & \makecell{\mathrm{if~}d=0:~ br~(\mathrm{for~}\\ c=0,\mathrm{~see~Figure~}\ref{kl2}) \\ \mathrm{if~}d\ge1:~ (rrb)^{d-1}rrrbr} & (rrb)^{d}rrb \\
    \hline
    \makecell{k=3c+3 ~ (c\in \mathbb{N}) \\ \mathrm{let~the~color~of~} \\ (w_1,w_2,\dots,w_k) \\ \mathrm{be~} (rrb)^{c}rrb} & (rrb)^{d}r & \makecell{\mathrm{if~}d=0:~ br \\ \mathrm{if~}d\ge1:~ (rrb)^{d-1}rrrbr} &  \makecell{\mathrm{if~}d=0:~ brr \\ \mathrm{if~}d\ge1:~ (rrb)^{d-1}rrrbrr} \\
    \hline
\end{array}
\]
\end{center}
\label{tab:multicol}
\caption{The crumby colorings we use if $d(v)=3$, depending on the values of $k$ and $\ell$.}
\end{table}


It is immediate that these are crumby colorings, and have \pa, thus the coloring process can start.

\bigskip

{\bf Adding a new ear:} Let us assume that $G_i$ has already been colored and \pa\ holds.
We consider the next ear $E_{i+1}=(x,z_1,z_2,\dots,z_{\ell},y)$ of the ear decomposition.
\pa\  implies that the color of the endpoints of this ear cannot be both blue.
We assume $x$ is red. In the following case analysis we indicate the number of internal points of $E_{i+1}$, and the colors of its endpoints.

{\it Case $\ell=1$, $rb$:} If $x$ is not an endpoint of any red $P_3$, then we color $z_1$ red.
If $y$ is a singleton blue vertex, then we color $z_1$ blue. Otherwise, we interchange the color of $x$ and $y$, their previous components remained admissible, and color $z_1$ red.

{\it Case $\ell=1$, $rr$:} We color $z_1$ blue.

{\it Case $\ell=2$, $rb$:}
In the following table, we summarize the possibilities and give a suitable coloring for the endpoints and the internal points of the ear.



\[
\begin{array}{|c|c|c|c|c|}
    \hline
    \ell=2,~rb: & \makecell{x \mathrm{~is~an~endpoint} \\ \mathrm{of~a~red~} P_3 \\ \& \\ y \mathrm{~is~a~singleton} \\ \mathrm{~blue~vertex}} & \makecell{x \mathrm{~is~an~endpoint} \\ \mathrm{of~a~red~} P_3 \\ \& \\ y \mathrm{~is~in~a~blue} \\ K_2 \mathrm{~component}} & \makecell{x \mathrm{~is~not~an~endpoint} \\ \mathrm{of~a~red~} P_3 \\ \& \\ y \mathrm{~is~a~singleton} \\ \mathrm{~blue~vertex}} & \makecell{x \mathrm{~is~not~an~endpoint} \\ \mathrm{of~a~red~} P_3 \\ \& \\ y \mathrm{~is~in~a~blue} \\ K_2 \mathrm{~component}} \\
    \hline
    \makecell{\mathrm{color~of~} \\ (x,z_1,z_2,y)} & brrb & brrr & rrbb & rrbr \\
    \hline
\end{array}
\]

{\it Case $\ell=2$, $rr$:} If both $x$ and $y$ are red, then at most one of them can be an endpoint of a red $P_3$.
We may assume that there is no red $P_3$ in $G_i$ ending in $x$. We color $z_1$ red and $z_2$ blue, respectively.

{\it Case $\ell=3$, $rb$:} We color $z_1,z_2,z_3$ to $brr$.

{\it Case $\ell=3$, $rr$:} We may assume $x$ is not an endpoint of any red $P_3$.
If there is no (later) ear with endpoints $z_2$ and $z_3$, then we color $z_1,z_2,z_3$ to $rbb$ maintaining \pa.
On the other hand, if there exists an ear with endpoints $z_2$ and $z_3$ together with $m$ internal points (denote them by
$w_1,\dots,w_m$), then we merge the two ears and add them to
$G_i$ in one step.
We give a crumby coloring of the resulting graph in the following table.
Independent of the value $m$, we color $z_1$ and $z_3$ blue in order to avoid conflicts with the rest of the coloring.
We color $z_2$ red as well as $w_1$.
The coloring of $(w_1,w_2,\dots,w_m)$ is only shown for $m\le 6$ in the following table.
For greater $m$, we use the crumby coloring for $m-3$ and add $brr$ at the end.

\[
\begin{array}{|c|c|c|c|c|c|c|}
    \hline
    \ell=3,~rr: & m=1 & m=2 & m=3 & m=4 & m=5 & m=6 \\
    \hline
    \makecell{\mathrm{color~of~} \\ (w_1,w_2,\dots,w_m)} & r & rr & rrb & rbrr & rrbrr & rrbrrr \\
    \hline
\end{array}
\]

{\it Case $\ell=4$, $rb$:} We color $z_1,z_2,z_3,z_4$ to $brrr$.

{\it Case $\ell=4$, $rr$:} We color $z_1,z_2,z_3,z_4$ to $brrb$.

{\it Case $\ell=5$, $rb$:} Depending on the type of the components of $x$ and $y$ we need to color the points of the next ear a bit differently. In the following table, we summarize the possibilities and give a suitable coloring for the endpoints and the internal points of the ear.

\[
\begin{array}{|c|c|c|c|}
    \hline
    \ell=5,~rb: & \makecell{x \mathrm{~is~an~endpoint} \\ \mathrm{of~a~red~} P_3 \\ \& \\ y \mathrm{~is~a~singleton} \\ \mathrm{~blue~vertex}} & \makecell{x \mathrm{~is~an~endpoint} \\ \mathrm{of~a~red~} P_3 \\ \& \\ y \mathrm{~is~in~a~blue} \\ K_2 \mathrm{~component}} & \makecell{x \mathrm{~is~not~an~endpoint} \\ \mathrm{of~a~red~} P_3} \\
    \hline
    \makecell{\mathrm{color~of~} \\ (x,z_1,z_2,z_3,z_4,z_5,y)} & brrbrrb & brrbrrr & rrbrrrb \\
    \hline
\end{array}
\]

{\it Case $\ell=5$, $rr$:} We color $z_1,z_2,z_3,z_4,z_5$ to $brrrb$.

{\it Case $\ell=6$, $rb$:} We color $z_1,z_2,z_3,z_4,z_5,z_6$ to $brrbrr$.

{\it Case $\ell=6$, $rr$:} We may assume $x$ is not an endpoint of any red $P_3$. We color\\ $z_1,z_2,z_3,z_4,z_5,z_6$ to $rbrrrb$.

For $\ell\ge 7$, we create a crumby coloring using the cases for smaller values of $\ell$.
We start by $brr$, and continue with the given coloring of the ear with $\ell-3$ internal points. By starting with $brr$, we trace back to a similar situation for $\ell-3$ internal points in which $z_3$ takes over the role of $x$.
We remark that $z_3$ is in a red $K_2$ component, thus cannot be an endpoint of a red $P_3$.
\end{proof}

A general outerplanar graph is not necessarily 2-connected. It is glued together from 2-connected blocks in a tree-like manner.
Some of the edges can form a tree hanging from a vertex of a block, or connecting a number of $2$-connected outerplanar components. In our case, the maximum degree 3 condition gives some extra structural information.
We are convinced that the natural extension of Theorem~\ref{t:outer2} to all subcubic outerplanar graphs holds.

\begin{conjecture} \label{outerplanar}
Every outerplanar graph with maximum degree $3$ admits a crumby coloring.
\end{conjecture}

Considering this problem, one gets the impression that particular small trees attached to the vertices of a 2-connected outerplanar graph make the difficulty. It turns out that most trees do not cause any problems at all.
To prove this statement, we need the following result.

\begin{theorem}\label{t:fa2foku}
Any subcubic tree $T$ admits a crumby coloring such that the color of an arbitrary vertex of degree $2$ is prescribed, unless $T=P_3$.
\end{theorem}

\begin{proof}
If $T=P_3$, then the middle vertex cannot be blue in a crumby coloring. Therefore, this is an exception. From now on, we assume that $T$ has at least 4 vertices. Every tree admits a crumby coloring by Theorem~\ref{t:treesplus}. Let us suppose that $T$ is a minimal example of a tree, which has a vertex $v$ of degree 2 such that in any crumby coloring of $T$, the color of $v$ must be red. We think of $v$ as the root, and denote the two neighbors of $v$ by $x$ and $y$.

If any of the neighbors of $v$ is of degree 2, say $x$, then we can delete the edge $vx$ and consider the two remaining trees $T_v$ rooted at $v$ and $T_x$ rooted at $x$. We get a contradiction by using Theorem~\ref{t:treesplus} with prescribed color red on $x$ and blue on $v$ in the respective trees.

\begin{figure}[!ht]
    \centering
    \includegraphics[width=0.6\linewidth]{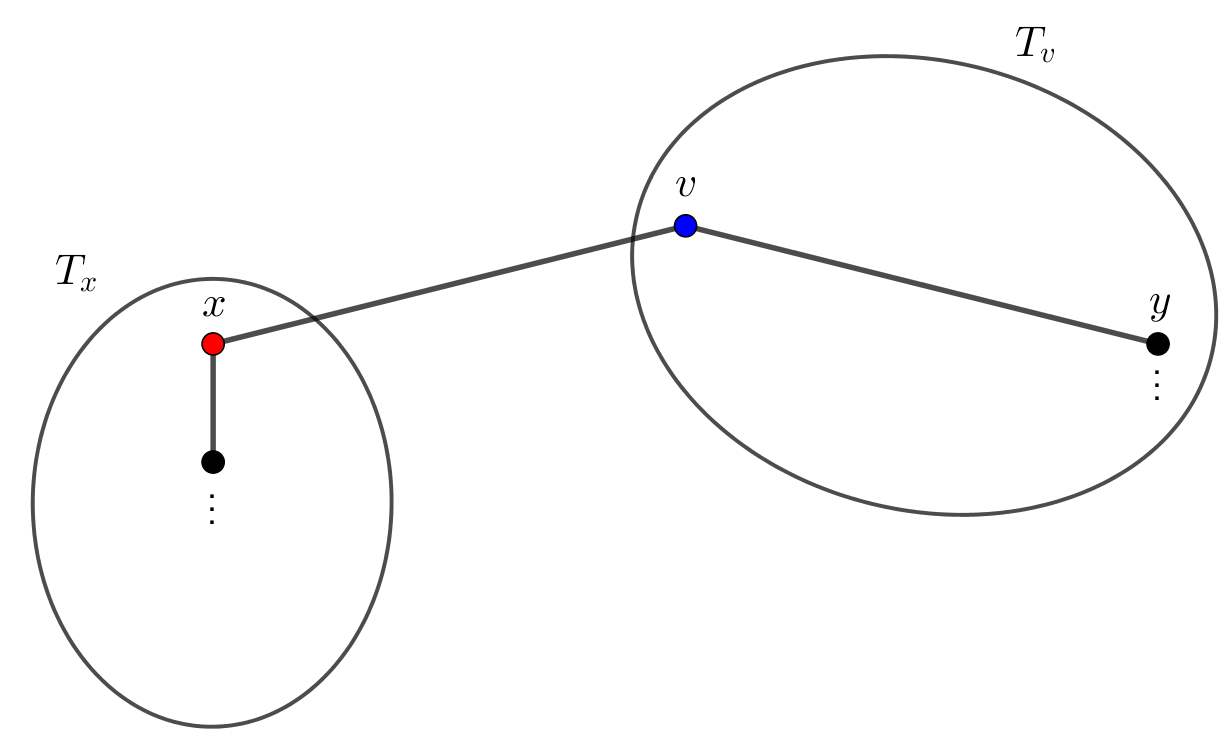}
    \caption{If $d_T(x)=2$, then we get a contradiction.}
    \label{f:fa2foku2}
\end{figure}

Since $T$ has at least 4 vertices, we may assume that $d_T(x)=3$. As before, we get a contradiction if the color of $x$ can be red in a crumby coloring of $T_x$, since we can color $v$ blue and use Theorem~\ref{t:treesplus} on $T_v$. Therefore, let us suppose that $T_x$ is a tree, for which the degree 2 vertex $x$ can only be colored blue in a crumby coloring. Denote the neighbors of $x$ in $T_x$ by $z$ and $w$.

Due to the same reasons as above, the degree of $z$ and $w$ cannot be 2 in $T_x$. It cannot be 1 either, since in that case $T_x$ has a crumby coloring in which the color of that leaf is prescribed red. Consequently $x$ is also red, which is a contradiction. Hence $d_{T_x}(z)=d_{T_x}(w)=3$, and by the minimality of $T$, we know that $T_z$ admits a crumby coloring such that the degree 2 vertex $z$ is blue. Now we may delete the edge $xz$ and precolor the degree 1 vertex $x$ red and find a crumby coloring of a subgraph of $T_x$. However, we can add back the edge $xz$ giving a crumby coloring of $T_x$ with red $x$, a contradiction. The same holds for $T_w$, but there is one exception: if both $T_z=T_w=P_3$. In Figure~\ref{f:fa2foku}, we give a crumby coloring of $T_x$ so that $x$ is red, which concludes the proof.

\begin{figure}[!ht]
    \centering
    \includegraphics[width=0.6\linewidth]{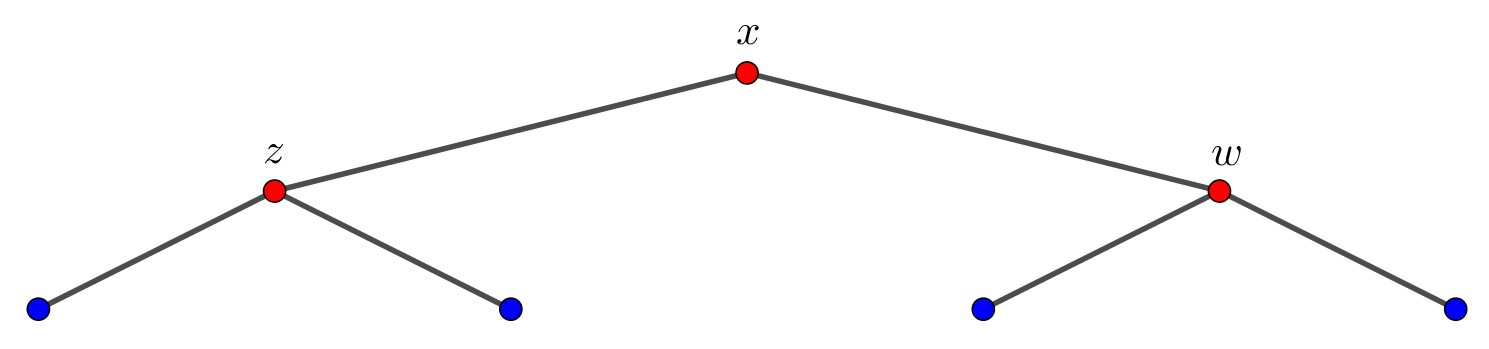}
    \caption{A crumby coloring of $T_x$ such that $x$ is red.}
    \label{f:fa2foku}
\end{figure}
\end{proof}

\begin{remark} \label{csakketfa}
If $G$ is a graph that admits a crumby coloring, and $T$ is an arbitrary tree with a leaf $v$, then let $G_T$ denote a graph which we get by identifying $v$ with any vertex of $G$.
Observe that if an attachment tree $T$ is not $K_2$ or $K_{1,3}$, then it is trivial to get a crumby coloring of $G_T$. The key idea is to assign different colors to $v$ and its only neighbor $x$ inside $T$. Consider a crumby coloring of $G$, therefore the color of $v$ is given, and color $x$ differently. By Theorem \ref{t:treesplus} and Theorem \ref{t:fa2foku} (depending on $d_T(x)$), we can extend this coloring to a crumby coloring of $T-v$ which results in a crumby coloring of $G_T$.

Therefore, it is indifferent with respect to crumby colorings to attach trees, which are not isomorphic to $K_2$ or $K_{1,3}$. In the sequel, we assume that every attachment tree is either $K_2$ or $K_{1,3}$.
\end{remark}

\bigskip

Now, we prove a basic instance of Conjecture~\ref{outerplanar} relying on Theorems~\ref{t:treesplus} and \ref{t:fa2foku}.

\begin{proposition}\label{p:korbarhollelog}
Let $C$ be a cycle with vertices $v_1,\dots,v_k$, plus we might attach arbitrary trees $\{T_i\}$ to vertices $\{v_i\}$ of $C$, where $i\in I$ and $I\subseteq [k]$. The resulting graph $G$ admits a crumby coloring.
\end{proposition}

\begin{proof}
We may assume that each attachment tree is isomorphic to $K_2$ or $K_{1,3}$ by Remark~\ref{csakketfa}.
Our arguments slightly vary depending on some properties of $G$, thus we explain them separately.

Notice that some vertices of $C$ have attachments and some do not.
In the latter case, the vertex is called empty. First, let us assume that there are no empty vertices at all. 

We notice that the case where $k$ is even is simple.
We color the vertices of $C$ alternately red and blue.
This gives the prescribed color of a leaf $v_i$ in the tree $T_i$. We color $T_i$ using Theorem~\ref{t:treesplus} for each $i=1,\dots,k$.
These colorings together form a crumby coloring of $G$.

Assume now that $k$ is odd.
We try to reuse the previous strategy by cutting off two consecutive vertices $v_i$ and $v_{i+1}$ and the trees $T_i$ and $T_{i+1}$ from $G$.
We notice that the remaining graph $H$ admits a crumby coloring by the previous argument. In particular, the first and last vertices ($v_{i+2}$ and $v_{i-1}$) on $C-\{v_i,v_{i+1}\}$ receive the same color.

For every $j$ between 1 and $k$, the tree $T_j-v_j$ admits a crumby coloring.
Let us record for every $j$ the color of $u_j$, the neighbor of $v_j$ in $T_j$.
Since $k$ is odd, there is an index $\ell$ such that $u_{\ell}$ and $u_{\ell+1}$ received the same color, say blue.
Now we color $v_{\ell}$ and $v_{\ell+1}$ red and cut the cycle $C$ by removing $\{v_{\ell},v_{\ell+1}\}$. We color $H$ as before such that we color the first and last vertex on $C-\{v_{\ell},v_{\ell+1}\}$ blue.
If $u_{\ell}$ was red, then we interchange colors accordingly.
Altogether, a crumby coloring of $G$ arises.

Unless there are no attachment trees at all (which case is easy), we can find two consecutive vertices of $C$, say $v_1$ and $v_2$ such that there is a tree attached to $v_1$, but $v_2$ has none. We use the following algorithm to color the vertices on $C$ starting by coloring $v_1$ red and $v_2$ blue. Our aim is to color the vertices along $C$ alternately, except in one case, when after a blue vertex we color an empty vertex red.
In that case, the next vertex must be also red.
Observe that if a red vertex is non-empty, then no matter if the tree is $K_2$ or $K_{1,3}$, we can color its vertices maintaining the crumby property. If $v_{i-1}$ is blue, and $v_i$ is an empty red, then $v_{i+1}$ must also be red.
However, it is attainable that $v_{i+1}$ is not an end of a red $P_3$. Only two problems can occur during this algorithm.
Both of them might happen, when we color $v_k$.

If $v_k$ was blue, then $v_1$ might remain a red singleton. However, this cannot be the case by the existence of $T_1$. Otherwise if $v_k$ is red, then we might create a large red component.
If $T_1=K_2$, then the leaf of $T_1$ can be blue.
Hence the red component cannot contain a red $P_4$, since $v_k$ was not an end of a red $P_3$. If $T_1=K_{1,3}$, then the center of $T_1$ must be red, which causes a problem if $v_{k-1}$ is an empty red or $T_k=K_{1,3}$. If we created a red $P_4$, then we recolor $v_1$ to blue and color the remaining vertices in $T_1$ red.
\end{proof}

\begin{remark}
Using the ideas of the previous proofs, we can prove Conjecture~$\ref{outerplanar}$ for a few other classes.
For instance, if $G$ is glued together from $2$-connected pieces in a tree-like fashion by single edges or paths of any length. Actually the paths might be replaced by any tree, as long as the first vertex outside of a $2$-connected piece has degree $2$. Even if the degree is $3$, our algorithm works except when the tree part between two $2$-connected components is precisely $P_3$, see Figure~$\ref{f:fa_p3}$.
In these good cases, we use Theorem~$\ref{t:treesplus}$ and Theorem~$\ref{t:fa2foku}$ as follows.
We first color a $2$-connected outerplanar subgraph $G_1$.
There is at least one vertex of attachment $x_1$, where a tree $T$ is glued on $G_1$.
Let $y_1$ be the neighbor of $x_1$ in $T$, which we know has degree $1$ or $2$ in $T-x_1$.
We prescribe the color of $y_1$ to be different from that of $x_1$.
We continue this way, until the entire graph is colored.
\end{remark}

\begin{figure}[!ht]
    \centering
    \includegraphics[width=0.6\linewidth]{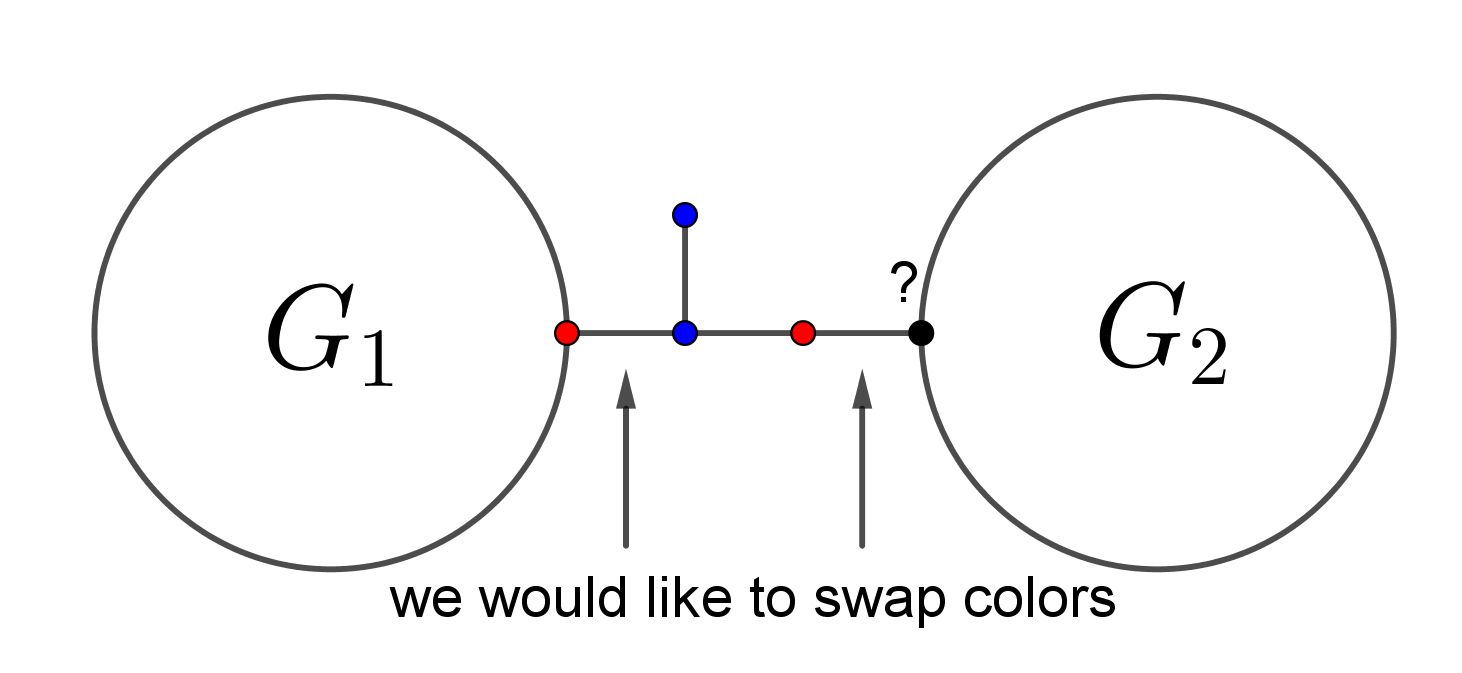}
    \caption{Problematic situation between two 2-connected outerplanar components $G_1$ and $G_2$.}
    \label{f:fa_p3}
\end{figure}

\section{Subdivisions of the complete graph on 4 vertices} \label{sec:k4}


Here we consider subdivisions of $K_4$, that has played interesting role in coloring problems \cite{j&t}.
As a strengthening of Hajós' conjecture, Toft \cite{bjarne} posed the problem if every 4-chromatic graph contains a totally odd subdivision of $K_4$.
Thomassen \cite{toks1} and independently Wang \cite{toks2} gave an affirmative answer for this.

Bellitto et al. \cite{counter} constructed planar graphs refuting Conjecture~\ref{ccc}.
Characteristically, those counterexamples have $K_4$-minors.
Therefore, we study whether this property has fundamental importance.
We conjecture that every $K_4$-minor-free subcubic graph possesses a crumby coloring.
On the other hand, we show the topological appearance of one copy of $K_4$ is not yet an obstacle.

As the core of the problem, we first consider $\le2$-subdivisions of $K_4$.
That is, every edge contains 0, 1 or 2 subdivision vertices.
It feels straightforward to give a computer-assisted proof, which we did.
We decided to include it as an appendix. However, we opted for a human proof argument.

\begin{lemma}\label{l:subsubk}
Let $G$ be a subdivision of $K_4$ such that every edge is divided into at most $3$ parts. The graph $G$ admits a crumby coloring.
\end{lemma}

\begin{proof}
Let $V(K_4)=\{A,B,C,D\}$.
Every edge of $K_4$ may remain intact or might be subdivided by either one or two new internal vertices.
Our arguments are organized by the number of intact edges.

If there are no intact edges (genuine subdivision), then color the vertices of $K_4$ red and every subdivision vertex blue.
Since the red vertices are isolated, we must recolor some internal vertices red.
If there are two independent edges of $K_4$ with one internal vertex each, then recolor these internal vertices red.
Otherwise, there exists a vertex of $K_4$, vertex $B$ say, with at least two incident original edges $BC$ and $BD$ with two internal vertices.
There are two possibilities regarding the number of internal points on $AD$ and $AC$ as one can see in Figure \ref{subk0}.
If there is only one internal vertex on one of these edges, then change its color to red, and change the color of two internal vertices on $BD$ and $BC$ as shown. In the other case, one can create a crumby coloring just like in the picture on the right. Note that every dashed edge have at least one internal vertex and these are blue. Hence we got a crumby coloring.

\vspace{-3pt}
\begin{figure}[!ht]
    \centering
    \includegraphics[width=0.7\linewidth]{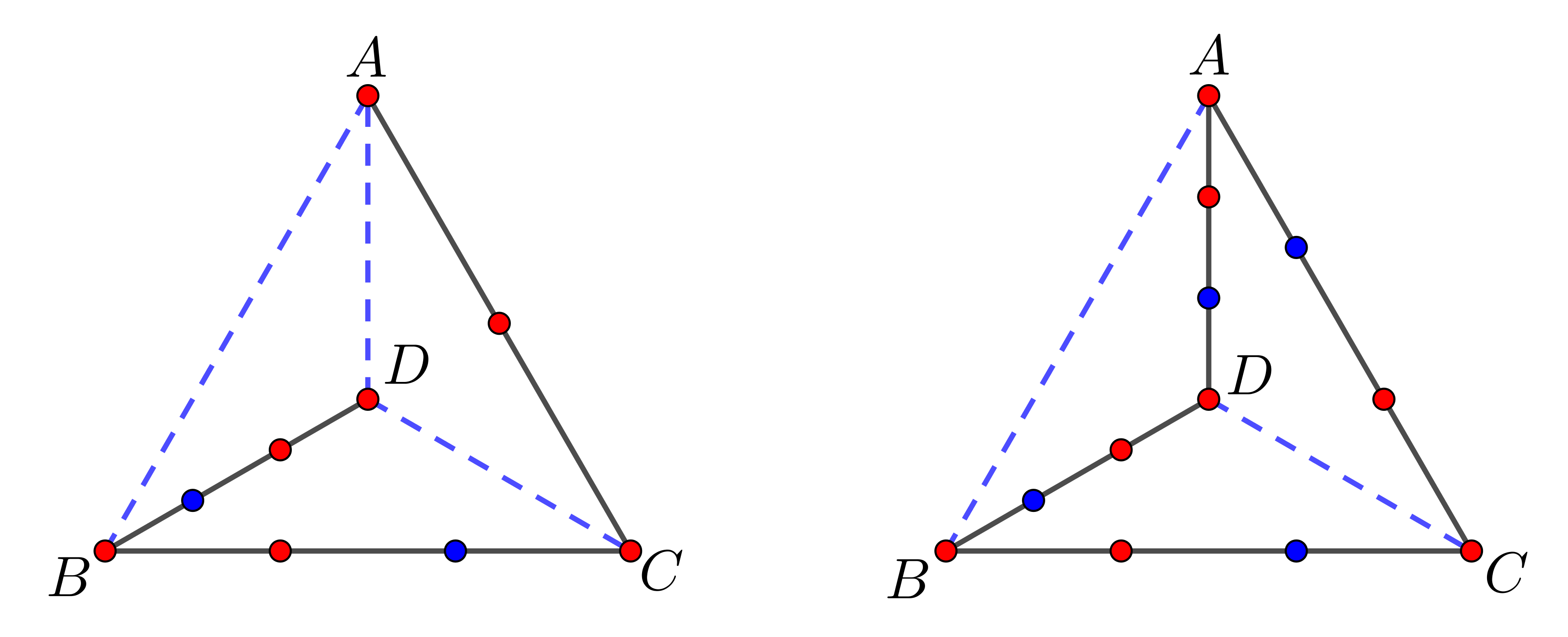}
    \caption{Crumby colorings of genuine subdivisions of $K_4$, where vertex $B$ is incident with 2 edges with 2 internal vertices.}
    \label{subk0}
\end{figure}

Next assume there is precisely one intact edge, $AB$ say.
If $CD$ contains exactly one internal vertex $x$, then color $x$ and the vertices of $K_4$ red, and color the other internal vertices blue to get a crumby coloring.
Thus we may assume that $CD$ contains two internal vertices.
If any of the remaining 4 edges of $K_4$ contains two internal vertices, then there is a path on 7 vertices formed along these two particular edges of $K_4$.
We color the vertices of the path $rrbrrbr$ starting from $C$.
It can be extended to a crumby coloring by coloring the vertices of $K_4$ red and the internal vertices blue. In Figure \ref{subk1}, we illustrate such a coloring and also cover the only remaining case. Again all dashed edges contain blue internal vertices.

\begin{figure}[!ht]
    \centering
    \includegraphics[width=0.7\linewidth]{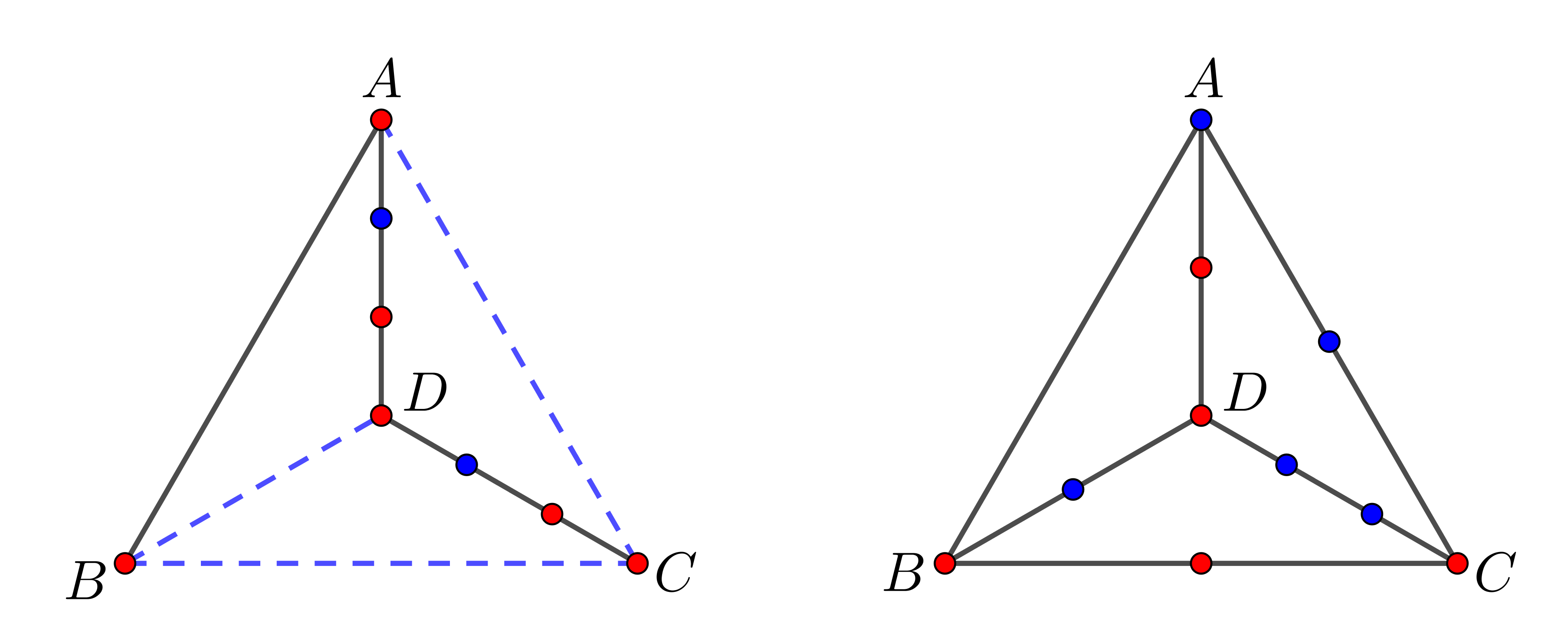}
    \caption{Extendable coloring of the path and a crumby coloring of the remaining case.}
    \label{subk1}
\end{figure}

Suppose there are exactly two intact edges of $K_4$.
If these two edges are independent, then again color the vertices of $K_4$ red and the internal points blue.
This results in a crumby coloring.
Therefore, we may assume that the intact edges are $AB$ and $BD$. If one of the edges incident to $C$ contains two internal vertices, then color the internal vertex adjacent to $C$ red along with $A,B,C,D$. Color the other vertices blue. In Figure \ref{subk2}, we give crumby colorings in the case, where the three edges incident to $C$ have 1 subdivision vertex. There are two such cases depending on the number of internal vertices on $AD$.

\begin{figure}[!ht]
    \centering
    \includegraphics[width=0.7\linewidth]{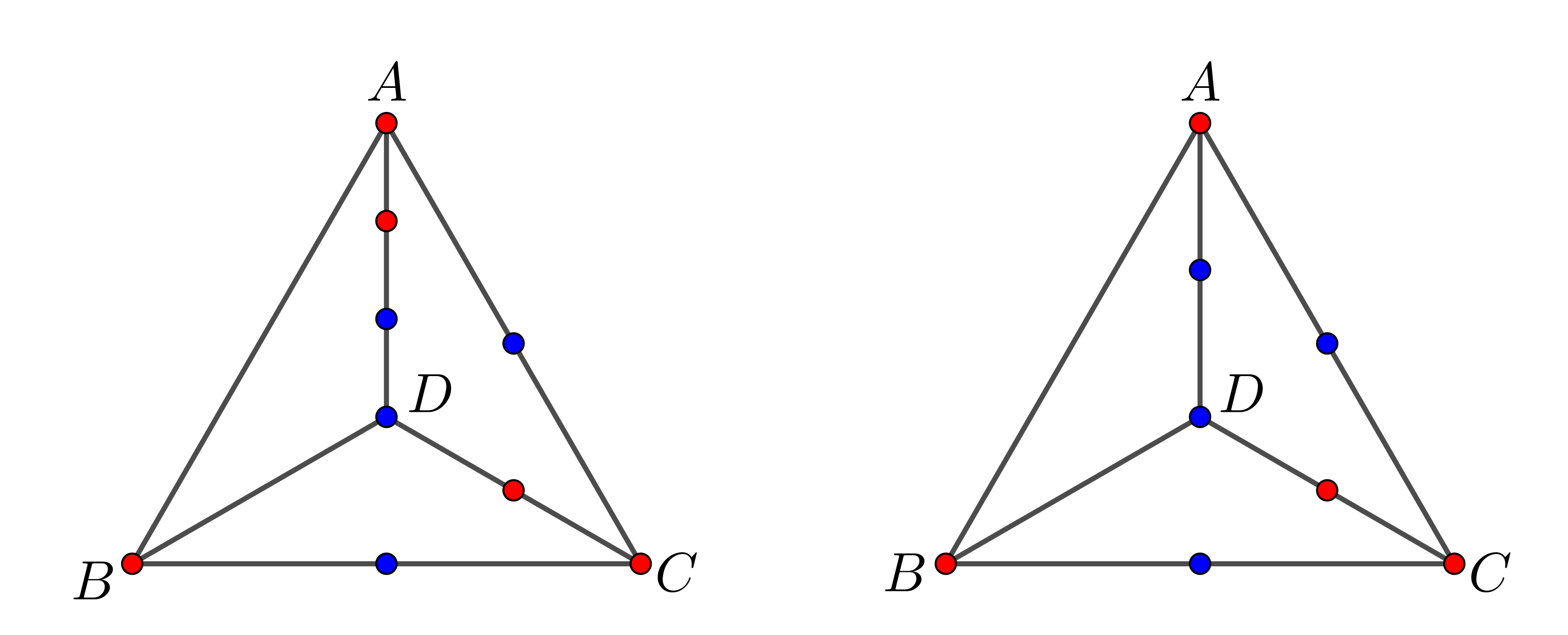}
    \caption{Two intact edges and all three edges incident to $C$ have exactly one subdivision vertex.}
    \label{subk2}
\end{figure}

Assume there are at least three intact edges.
There might be three of them, which form a path on the vertices of $K_4$ or there are exactly three intact edges either incident with the same vertex of $K_4$ or forming a triangle. In Figure~\ref{subk3}, we give crumby colorings for the latter two cases by coloring the internal vertices on the dotted edge red, and on the dashed edges blue.

\begin{figure}[!ht]
    \centering
    \includegraphics[width=0.7\linewidth]{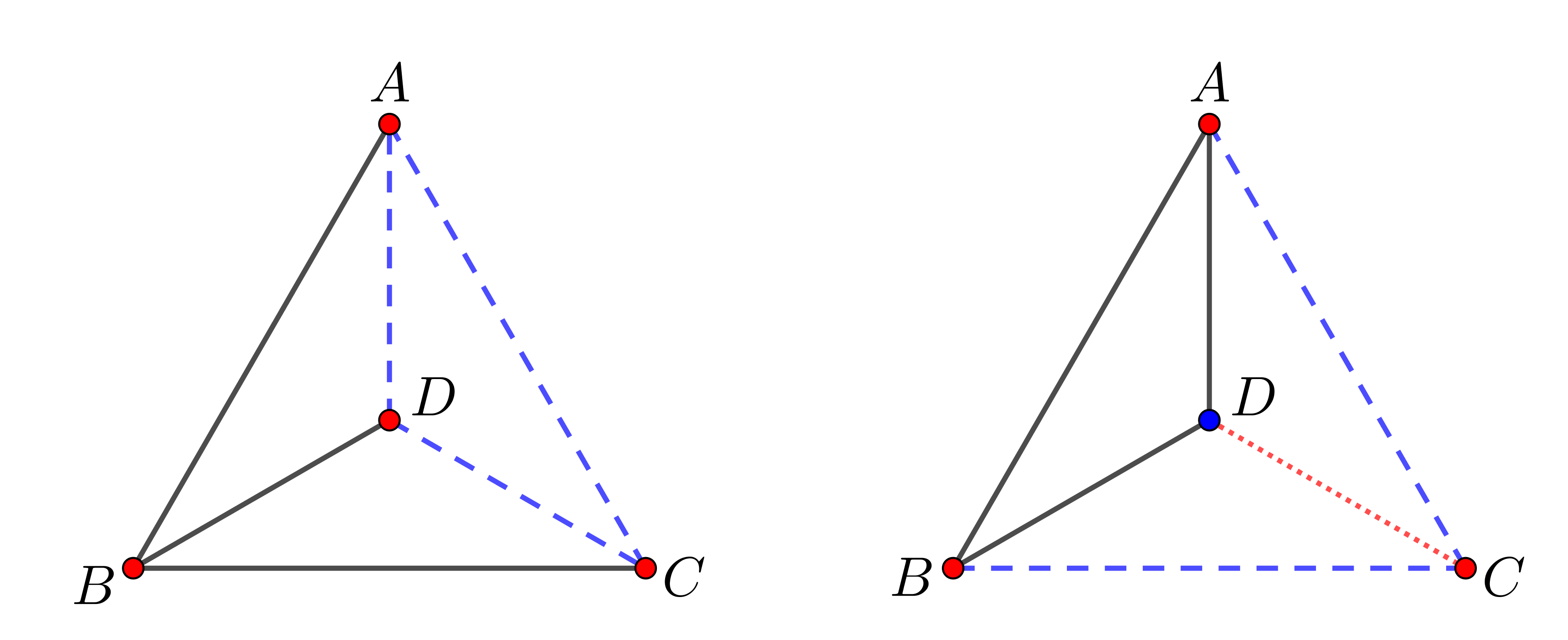}
    \caption{Precisely three intact edges incident to the same vertex or forming a triangle.}
    \label{subk3}
\end{figure}

Let us suppose that there is a path on the vertices of $K_4$, which consists of three intact edges.
We may assume that these are $AB$, $BC$ and $CD$.
Our idea is to color $A$ and $C$ red and $B$ blue (the color of $D$ might vary), and depending on the number of internal vertices on the remaining edges (which again form a path) color the vertices of this path suitably. Let $(i,j,k)$ denote the case, in which the number of internal vertices on $CA$, $AD$ and $DB$ are exactly $i$, $j$ and $k$, in that order.
There are three subcases depending on the value of $i$. In the following tables, we summarize the possibilities and give crumby colorings.


\begin{table}[h!] \label{tablei2}
            \footnotesize
                \begin{tabular}[h!]{|c|c|}
                    \hline
                    \makecell{$i=2$, \\ $(j,k)$:}&\makecell{\textrm{color~of~the~path} \\ $C-A-D-B$}\\ \hline
                    (0,0)& $R~bb~R~R~B$\\ \hline
                    (0,1)& $R~bb~R~R~b~B$\\ \hline
                    (0,2)& $R~bb~R~R~rb~B$ \\ \hline
                \end{tabular}
                \hfill
                \begin{tabular}[h!]{|c|c|}
                    \hline
                    \makecell{$i=2$, \\ $(j,k)$:}&\makecell{\textrm{color~of~the~path} \\ $C-A-D-B$}\\ \hline
                    (1,0)& $R~br~R~b~R~B$ \\ \hline
                    (1,1)& $R~br~R~b~R~b~B$ \\ \hline
                    (1,2)& $R~br~R~b~R~rb~B$ \\ \hline
                \end{tabular}
                \hfill
                \begin{tabular}[h!]{|c|c|}
                    \hline
                    \makecell{$i=2$, \\ $(j,k)$:}&\makecell{\textrm{color~of~the~path} \\ $C-A-D-B$}\\ \hline
                    (2,0)& $R~br~R~bb~R~B$ \\ \hline
                    (2,1)& $R~br~R~bb~R~b~B$ \\ \hline
                    (2,2)& $R~br~R~bb~R~rb~B$ \\ \hline
                \end{tabular}
                \caption{Crumby colorings if there are three undivided edges which form a path and $i=2$.}
\end{table}

\begin{table}[h!] \label{tablei1}
            \footnotesize
                \begin{tabular}[h!]{|c|c|}
                    \hline
                    \makecell{$i=1$, \\ $(j,k)$:}&\makecell{\textrm{color~of~the~path} \\ $C-A-D-B$}\\ \hline
                    (0,0)& $R~b~R~R~B$\\ \hline
                    (0,1)& $R~b~R~R~b~B$\\ \hline
                    (0,2)& $R~b~R~R~rb~B$ \\ \hline
                \end{tabular}
                \hfill
                \begin{tabular}[h!]{|c|c|}
                    \hline
                    \makecell{$i=1$, \\ $(j,k)$:}&\makecell{\textrm{color~of~the~path} \\ $C-A-D-B$}\\ \hline
                    (1,0)& $R~b~R~b~B~R$ \\ \hline
                    (1,1)& $R~b~R~b~B~r~R$ \\ \hline
                    (1,2)& $R~r~B~b~R~bb~R$ \\ \hline
                \end{tabular}
                \hfill
                \begin{tabular}[h!]{|c|c|}
                    \hline
                    \makecell{$i=1$, \\ $(j,k)$:}&\makecell{\textrm{color~of~the~path} \\ $C-A-D-B$}\\ \hline
                    (2,0)& $R~b~R~rb~R~B$ \\ \hline
                    (2,1)& $R~b~R~rb~R~b~B$ \\ \hline
                    (2,2)& $R~b~R~rb~R~rb~B$ \\ \hline
                \end{tabular}
                \caption{Crumby colorings if there are three undivided edges which form a path and $i=1$.}
\end{table}

\begin{table}[h!] \label{tablei0}
            \footnotesize
                \begin{tabular}[h!]{|c|c|}
                    \hline
                    \makecell{$i=0$, \\ $(j,k)$:}&\makecell{\textrm{color~of~the~path} \\ $C-A-D-B$}\\ \hline
                    (0,0)& $B~R~R~R$\\ \hline
                    (0,1)& $B~R~R~b~R$\\ \hline
                    (0,2)& $B~R~R~bb~R$ \\ \hline
                \end{tabular}
                \hfill
                \begin{tabular}[h!]{|c|c|}
                    \hline
                    \makecell{$i=0$, \\ $(j,k)$:}&\makecell{\textrm{color~of~the~path} \\ $C-A-D-B$}\\ \hline
                    (1,0)& $B~R~b~R~R$ \\ \hline
                    (1,1)& $R~B~r~R~r~B$ \\ \hline
                    (1,2)& $B~R~b~R~rb~R$ \\ \hline
                \end{tabular}
                \hfill
                \begin{tabular}[h!]{|c|c|}
                    \hline
                    \makecell{$i=0$, \\ $(j,k)$:}&\makecell{\textrm{color~of~the~path} \\ $C-A-D-B$}\\ \hline
                    (2,0)& $B~R~bb~R~R$ \\ \hline
                    (2,1)& $B~R~br~R~b~R$ \\ \hline
                    (2,2)& $B~R~br~R~rb~R$ \\ \hline
                \end{tabular}
                \caption{Crumby colorings if there are three intact edges which form a path and $i=0$.} 
\end{table}

This finishes all cases of the lemma.
\end{proof}

Using the solutions for the restricted cases in the previous lemma, we can obtain a crumby coloring for all subdivisions of $K_4$.

\begin{theorem}\label{t:subk}
Let $G$ be a subdivision of $K_4$. The graph $G$ admits a crumby coloring.
\end{theorem}

\begin{proof}
Let $G'$ be the following reduction of $G$.
Independently, on each edge of $K_4$,
we replace the $k$ subdivision vertices by $k$ $\mathrm{mod}~3$ vertices.
By Lemma~\ref{l:subsubk}, there is a crumby coloring of $G'$.
%
We extend the coloring of $G'$ independently on each edge of $K_4$. If along an edge of $K_4$ both colors appear, then we can find two consecutive vertices $x$ and $y$ with different colors.
We insert the necessary number of blocks of $brr$ between $x$ and $y$ such that it remains a crumby coloring.

Otherwise, the considered edge of $K_4$ is monochromatic.
If every vertex is red, then we insert $rbr$ between any two of them.  Now we continue the extension (if necessary) just like in the previous case, since there exist consecutive vertices with different colors.
A monochromatic blue edge means two blue vertices in $K_4$.
In that case, we insert $rrr$ between them.
Again, we continue the extension (if necessary) just like in the first case, since there exist consecutive vertices with different colors.
\end{proof}

In a slightly opposite direction, motivated by our proof idea of Theorem~\ref{t:outer2}, we pose the following

\begin{conjecture}
 Every $K_4$-minor-free graph admits a crumby coloring.
\end{conjecture}

\section*{Acknowledgements}
The first author would like to thank Bjarne Toft for 25 years of friendship and encouragement. He influenced our work in Section~\ref{sec:k4}. The first author was partially supported by ERC Advanced Grant "GeoScape" and NKFIH Grant K. 131529.
The second author was supported in part by the Hungarian National Research, Development
and Innovation Office, OTKA grant no. SNN 132625 and K 124950.

\newpage
\appendix
\section{Appendix}

\begin{figure}[!h]
    \centering
    \includegraphics[width=0.7\linewidth]{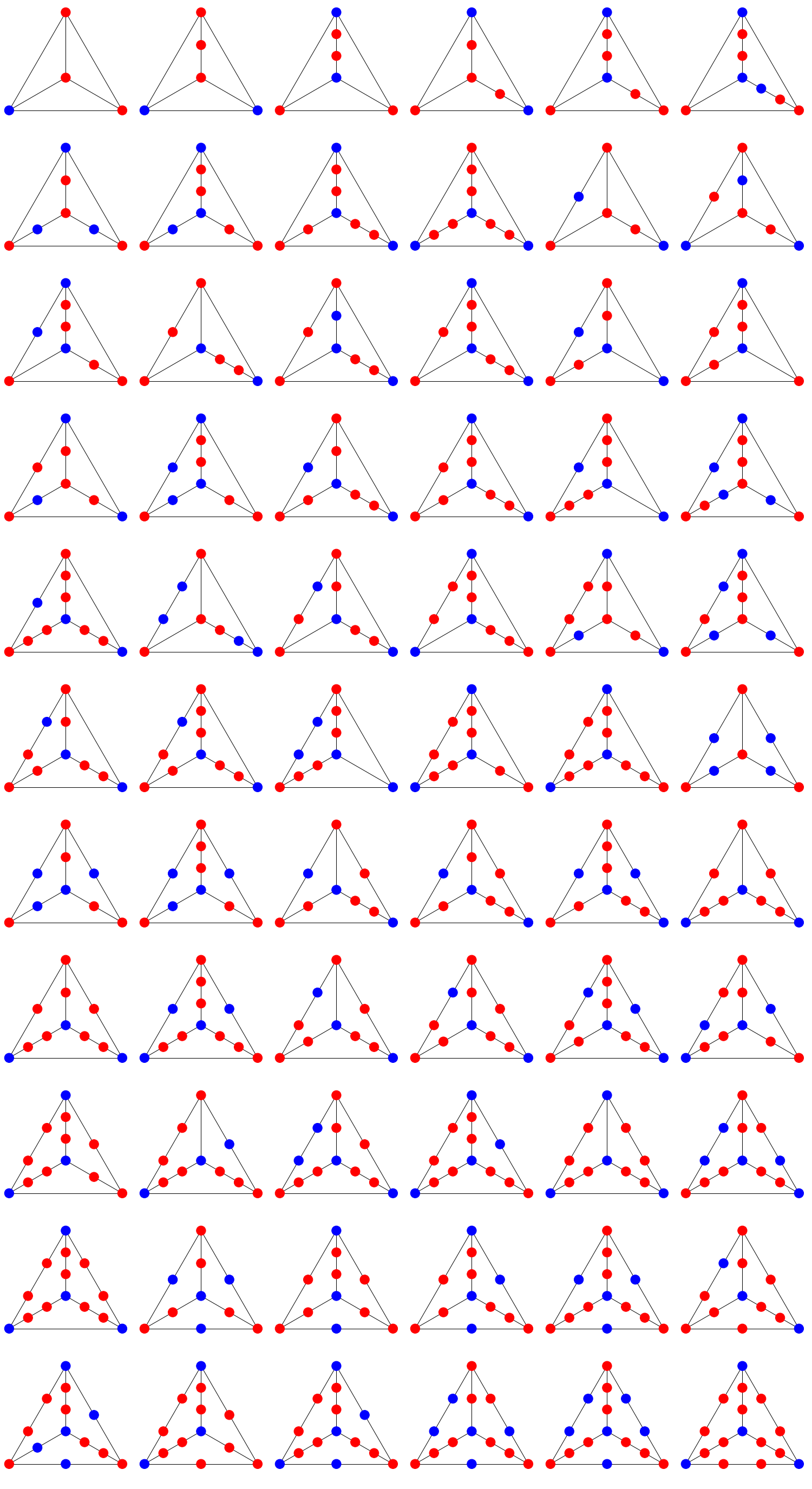}
    \caption{Computer-assisted proof of Lemma \ref{l:subsubk}.}
    \label{comp_k4}
\end{figure}

\end{document}